\documentclass{article}
\usepackage{amssymb}
\usepackage{amsmath}
\usepackage{amsthm}
\usepackage{amsfonts}
\begin{document}
\def\mathscr{\mathcal}
\topmargin= -.2in \baselineskip=15pt
\newtheorem{theorem}{Theorem}[section]
\newtheorem{proposition}[theorem]{Proposition}
\newtheorem{lemma}[theorem]{Lemma}
\newtheorem{corollary}[theorem]{Corollary}
\newtheorem{conjecture}[theorem]{Conjecture}
\newtheorem{example}[theorem]{Example}
\theoremstyle{remark}
\newtheorem{remark}[theorem]{Remark}

\title {A Thom-Sebastiani Theorem in Characteristic $p$
\thanks{I would like to thank L. Illusie for informing me of unpublished
work of Deligne on the Thom-Sebastiani theorem, and thank P. Deligne
who made many suggestions and comments on earlier versions of this
paper. I am also thankful for the referee who points out many subtle
points that I ignored. My research is supported by the NSFC.}}

\author {Lei Fu\\
{\small Chern Institute of Mathematics and LPMC, Nankai University,
Tianjin 300071, P. R. China}\\
{\small leifu@nankai.edu.cn}}

\date{}
\maketitle

\begin{abstract}
Let $k$ be a perfect field of characteristic $p$, $X_i$ $(i=1,2)$
smooth $k$-schemes, $f_i:X_i\to\mathbb A_k^1$ two $k$-morphisms of
finite type, and $f:X_1\times_k X_2\to \mathbb A_k^1$ the morphism
defined by $f(z_1,z_2)=f_1(z_1)+f_2(z_2)$. For each $i\in\{1,2\}$,
let $x_i$ be a $k$-rational point in the fiber $f_i^{-1}(0)$ such
that $f_i$ is smooth on $X_i-\{x_i\}$. Using the $\ell$-adic Fourier
transformation and the stationary phase principle of Laumon, we
prove that the vanishing cycles complex of $f$ at $x=(x_1,x_2)$ is
the convolution product of the vanishing cycles complexes of $f_i$
at $x_i$ $(i=1,2)$.



\end{abstract}

\section*{Introduction}

Let $f_i:(\mathbb C^{n_i},0)\to (\mathbb C,0)$ $(i=1,2)$ be two
germs of analytic functions with isolated critical points, and
consider the germ $f:(\mathbb C^{n_1+n_2},0)\to(\mathbb C,0)$
defined by $f(z_1,z_2)=f_1(z_1)+f_2(z_2)$. The classical
Thom-Sebastiani Theorem (\cite{TS}) states that the vanishing cycles
complex of $f$ is isomorphic to the tensor product of those of $f_1$
and $f_2$. This theorem is not correct for a general field $k$.
Indeed, let $f:\mathbb A_k^n\to\mathbb A_k^1$ be the $k$-morphism
defined by the quadratic polynomial
$$f=z_1^2+\cdots+ z_n^2$$ for a field $k$ of characteristic $\not=2$,
$\ell$ a prime number distinct from the characteristic of $k$,
$\overline{\mathbb Q}_\ell$ an algebraic closure of the field
$\mathbb Q_\ell$ of $\ell$-adic numbers, and
$R\Phi_f(\overline{\mathbb Q}_\ell)$ the vanishing cycles complex
for the constant $\ell$-adic sheaf $\overline{\mathbb Q}_\ell$
relative to the morphism $f$. By \cite[XV 2.2]{DK}, we have
$R^i\Phi_f(\overline{\mathbb Q}_\ell)=0$ for $i\not=n-1$, and
$R^{n-1}\Phi_f(\overline{\mathbb Q}_\ell)$ is supported at the
origin. Let $K=k((t))$ be the formal Laurent series field. The stalk
$(R^{n-1}\Phi_f(\overline{\mathbb Q}_\ell))_{0}$ at the origin is a
one-dimensional $\overline{\mathbb Q}_\ell$-vector space on which
$\mathrm{Gal}(\overline K/K)$ acts continuously. By \cite[XV 2.2.5
D, E]{DK}, we have
\begin{eqnarray*}
\Big(R^{n-1}\Phi_{f}(\overline{\mathbb Q}_\ell)\Big)_{0}\cong
V_\chi\left(-\left[\frac{n}{2}\right]\right).
\end{eqnarray*}
for some one dimensional $\overline{\mathbb Q}_\ell$-vector space
$V_\chi$ on which $\mathrm{Gal}(\overline K/K)$ acts through a
character $$\chi:\mathrm{Gal}(\overline K/K)\to\overline{\mathbb
Q}_\ell^\ast$$ of order $2$, where $\left[\frac{n}{2}\right]$
denotes the largest integer that is less than or equal to
$\frac{n}{2}.$ Since $\left[\frac{n}{2}\right]$ in the Tate twist is
not linear in $n$, one sees immediately that the Thom-Sebastiani
theorem in terms of tensor product does not hold for a general field
$k$ and quadratic
$$f_1=z_1^2+\cdots+z_{n_1}^2,\quad
f_2=z_{n_1+1}^2+\cdots+z_{n_1+n_2}^2$$ if both $n_1$ and $n_2$ are
odd.

In \cite{DLetter} (unpublished), Deligne studies the analogue of the
Thom-Sebastiani theorem for the variation morphism, and realizes
that the tensor product in the Thom-Sebastiani theorem should be
replaced by the convolution product. Deligne constructs the
vanishing cycles complex $V_f$ and the monodromy $T_f$ associated to
$f$ with the aid of those $(V_{f_i}, T_{f_i})$ $(i=1,2)$ associated
to $f_i$. Topologically, $V_f$ retracts to the join $V_{f_1}\ast
V_{f_2}$. This should be translated using convolution.

In \cite{V}, using the Fourier transformation, the formula
$$\int e^{itf} = \int e^{itf_1}\int e^{itf_2},$$ and the asymptotic
expansions of such integrals for $t\to \infty$, Varchenko proves a
Thom-Sebastiani theorem for the Hodge spectrum. Inspired by
Varchenko's method, we prove a Thom-Sebastiani theorem in
characteristic $p$ using the Deligne-Fourier transformation.

The Thom-Sebastiani theorem has been studied extensively over
$\mathbb C$. See for example \cite{DN}, \cite{G}, \cite{N},
\cite{S}. Due to the use of the Deligne-Fourier transformation, our
method can not be direcly applied to the characteristic $0$ case.
But it suggests that one can use the Fourier transformation for
$D$-modules to study the Thom-Sebastiani theorem over $\mathbb C$.

\section{Main result}

Throughout this paper, $k$ is a perfect field of characteristic $p$,
and $\ell$ is a prime number distinct from $p$. We assume that for
any finite extension $k'$ of $k$, the groups
$H^i(\mathrm{Gal}(k'/k),\mathbb Z/\ell \mathbb Z)$ $(i\in \mathbb
N)$ are finite. For any scheme $X$ of finite type over $k$, the
category $D_c^b(X,\overline{\mathbb Q}_\ell)$ constructed in
\cite[1.1]{DWeil} is a triangulated category. Let $S$ be a henselian
trait of equal characteristic $p$ with generic point $\eta$ and
special point $s$ such that $k(s)=k$, let $f:X\to S$ be a morphism
of finite type, and let $K$ be an object in $D^b_c(X,
\overline{\mathbb Q}_\ell)$. We refer the reader to \cite[XIII]{DK}
for the definitions and properties of the nearby cycles complex
$R\Psi_f(K)$ and the vanishing cycles complex $R\Phi_f(K)$ relative
to the morphism $f$. We also denote $R\Psi_f$ and $R\Phi_f$ by
$R\Psi$ and $R\Phi$ for convenience.

\begin{lemma} \label{prelemma} Let $X\to S$ be a morphism of finite
type, let $x$ be a $k$-rational point in the special fiber $X_s$,
and let $K\in\mathrm{ob}\, D^b_c(X,\overline{\mathbb Q}_\ell)$.
Suppose $X-\{x\}\to S$ is smooth and the sheaves $\mathscr
H^q(K)|_{X-\{x\}}$ are lisse for all $q$.

(i) $R\Phi(K)|_{X_{\bar s}-\{x\}}=0$.

(ii) Suppose furthermore that $X$ is pure of dimension $n$ and
regular at $x$, and $K$ is a lisse sheaf. Then $R^i\Phi(K)$ vanishes
for $i\not=n-1$, and $R^{n-1}\Phi(K)$ is a skyscraper sheaf on
$X_{\bar s}$ supported at $x$.
\end{lemma}

\begin{proof} See \cite[2.10]{I1}. We give a proof here for completeness.
(i) follows from the smooth base change theorem.
Under the assumption of (ii), $K[n]$ is a perverse sheaf on
$X$. By \cite[4.6]{I}, $R\Phi(K[n])[-1]$ is perverse. Combined with
(i), we see that $R\Phi(K[n])[-1]$ is a perverse sheaf on $X_{\bar
s}$ supported at $x$. Our assertion follows.
\end{proof}

Let $\mathbb A^1_{(0)}$ be the henselization of $\mathbb A_k^1$ at
$0$, let $\eta_0$ be its generic point, and let
$j:\eta_0\hookrightarrow \mathbb A^1_{(0)}$ be the canonical open
immersion. We can identify a
$\mathrm{Gal}(\bar\eta_0/\eta_0)$-module with a sheaf on $\eta_0$.
Let $(\mathbb A_k^1\times_k\mathbb A_k^1)_{(0,0)}$ be the
henselization of $\mathbb A_k^1\times_k\mathbb A_k^1$ at $(0,0)$,
and let
$$\tilde p_1,\tilde p_2,\tilde a:(\mathbb A_k^1\times_k\mathbb
A_k^1)_{(0,0)}\to \mathbb A^1_{(0)}$$ be the morphisms induced by
the two projections $p_1,p_2:\mathbb A_k^1\times_k\mathbb A_k^1\to
\mathbb A_k^1$ and the addition $a:\mathbb A_k^1\times_k\mathbb
A_k^1\to \mathbb A_k^1$ of the algebraic group $\mathbb A_k^1$,
respectively. Let $V_1$ and $V_2$ be $\overline{\mathbb
Q}_\ell$-representations of $\mathrm{Gal}(\bar\eta_0/\eta_0)$, and
regard them as sheaves on $\eta_0$. By \cite[2.7.1.3]{L}, the
vanishing cycles complex $R\Phi_{\tilde a}(\tilde p_1^\ast
j_{!}V_1\otimes^L \tilde p_2^\ast j_{!}V_2)$ relative to the
morphism $\tilde a$ is nonzero only at $(0,0)$ and at degree $1$.
Recall that the convolution product of $V_1$ and $V_2$
(\cite[2.7.2]{L}) is the $\mathrm{Gal}(\bar\eta_0/\eta_0)$-module
$$V_1\ast V_2=R^1\Phi_{\tilde a}(\tilde p_1^\ast j_{!} V_1\otimes \tilde p_2^\ast
j_{!}V_2)_{(0,0)}.$$ If $\mathscr V_1$ and $\mathscr V_2$ are
objects in $D_c^b(\eta_0,\overline{\mathbb Q}_\ell)$, we define
their convolution product to be
$$\mathscr V_1\underline{\ast} \mathscr V_2=
R\Phi_{\tilde a}(\tilde p_1^\ast j_{!} \mathscr V_1\otimes^L \tilde
p_2^\ast j_{!}\mathscr V_2)_{(0,0)}.$$ If $\mathscr V_i$ $(i=1,2)$
are objects $D_c^b(\eta_0,\overline{\mathbb Q}_\ell)$ defined by the
$\mathrm{Gal}(\bar\eta_0/\eta_0)$-modules $V_i$ $(i=1,2)$,
respectively, then $(\mathscr V_1\underline{\ast} \mathscr V_2)[1]$
is the object in $D_c^b(\eta_0,\overline{\mathbb Q}_\ell)$ defined
by the $\mathrm{Gal}(\bar\eta_0/\eta_0)$-module $V_1\ast V_2$.

Our main result is the following:

\begin{theorem}\label{mainthm} Let $f_i:X_i\to \mathbb A_k^1$ $(i=1,2)$ be
two flat $k$-morphisms of finite type,
$K_i\in\mathrm{ob}\,D_c^b(X_i,\overline{\mathbb Q}_\ell)$,
$X=X_1\times_S X_2$, $K=K_1\boxtimes^L K_2$, and $f:X=X_1\times_k
X_2\to\mathbb A_k^1$ the morphism defined by
$$f(z_1,z_2)=f_1(z_1)+f_2(z_2).$$ For each $i\in\{1,2\}$, let $x_i$ be
a $k$-rational point in the fiber $f_i^{-1}(0)$. Suppose that $X_i$
is regular, $\mathscr H^q(K_i)$ are lisse for all $q$, and
$f_i|_{X_i-\{x_i\}}$ is smooth. Denote by $x$ the $k$-rational point
$(x_1,x_2)$ on $X$. Denote respectively by $R\Phi_{f}$ and
$R\Phi_{f_i}$ the vanishing cycles functors relative to the
morphisms
$$X\times_{\mathbb A_k^1}\mathbb A^1_{(0)}\to \mathbb A^1_{(0)},\quad
X_i\times_{\mathbb A_k^1}\mathbb A^1_{(0)}\to \mathbb A^1_{(0)}$$
obtained from $f$ and $f_i$ by base change.

(i) $X$ is regular, and $f|_{X-\{x\}}$ is smooth.

(ii) As objects in $D_c^b(\eta_0,\overline{\mathbb Q}_\ell)$, we
have a canonical isomorphism
$$
(R\Phi_{f_1}(K_1))_{x_1}\underline\ast(R\Phi_{f_2}(K_2))_{x_2}\cong
(R\Phi_{f}(K))_x.$$

(iii) Suppose furthermore that for each $i$, $K_i$ is a lisse sheaf.
Let $n_i=\mathrm{dim}\,\mathscr O_{X_i,x_i}$, and let $n=n_1+n_2$.
Then as $\mathrm{Gal}(\bar\eta_0/\eta_0)$-modules, we have a
canonical isomorphism
$$
(R^{n_1-1}\Phi_{f_1}(K_1))_{x_1}\ast(R^{n_2-1}\Phi_{f_2}(K_2))_{x_2}\cong
(R^{n-1}\Phi_{f}(K))_x.$$
\end{theorem}

\begin{example}\end{example}
By our construction, the canonical isomorphisms in Theorem 1.2
(ii)-(iii) are associative. Let's use Theorem \ref{mainthm} to
calculate the vanishing cycles complex for the quadratic morphism
$$f:\mathbb A_k^n\to \mathbb A_k^1, \quad f(z)=z_1^2+\cdots+z_n^2$$
in the case where $\mathrm{char}\,k\not=2$.

First consider the case where $n=1$. Then $f$ is a finite morphism.
By the proper base change theorem, we have
$$R\Phi_f(\overline {\mathbb Q}_\ell)\cong
R\Phi(f_\ast\overline{\mathbb Q}_\ell).$$ We have
$$f_\ast \overline {\mathbb Q}_\ell\cong \overline {\mathbb Q}_\ell
\bigoplus j_!\mathscr K_{\chi_2},$$ where $j:\mathbb
A_k^1-\{0\}\hookrightarrow \mathbb A_k^1$ is the canonical open
immersion and $\mathscr K_{\chi_2}$ is the Kummer sheaf on $\mathbb
A_k^1-\{0\}$ associated to the nontrivial character
$$\chi_2:\mu_2(k)\to \{\pm 1\}$$ of order $2$.
Let $K=k((t))$ be the formal Laurent series field, and let
$V_{\chi_2}$ be the one dimensional $\overline{\mathbb
Q}_\ell$-vector space on which $\mathrm{Gal}(\overline K/K)$ acts
through the character
$$\mathrm{Gal}(\overline K/K) \to \overline{\mathbb
Q}_\ell^\ast,\quad \sigma\mapsto
\chi_2\Big(\frac{\sigma(\sqrt{t})}{\sqrt{t}}\Big).$$ For the
vanishing cycles functor with respect to the identity morphism on
$\mathbb A^1_{(0)}$, we have
$$R\Phi(\overline{\mathbb Q}_\ell)=0, \quad R\Phi(j_!\mathscr
K_{\chi_2})=V_{\chi_2}.$$ Indeed, let $\eta$ (resp. $0$) be the
generic (resp. closed) point of $\mathbb A^1_{(0)}$. For any object
$K$ in $D_c^b(\mathbb A^1_{(0)},\overline{\mathbb Q}_\ell)$,
$R\Phi(K)$ can be identified with the cone of the specialization
morphism $K_{\bar 0}\to K_{\bar \eta}$. If $K=\overline{\mathbb
Q}_\ell$, the specialization morphism is an isomorphism and hence
$R\Phi(\overline{\mathbb Q}_\ell)=0$. If $K=j_!\mathscr K_{\chi_2}$,
we have $K_{\bar 0}=0$ and $K_{\bar \eta}\cong V_{\chi_2}$. So we
have $R\Phi(j_!\mathscr K_{\chi_2})=V_{\chi_2}$. We thus have
$$R^0\Phi_f(\overline {\mathbb Q}_\ell)\cong V_{\chi_2}$$ for the
quadratic morphism $f(z)=z^2$. By Theorem 1.2, for $f=z_1^2+\cdots
+z_n^2$, we have $$(R\Phi_f^{n-1}(\overline {\mathbb
Q}_\ell))_0\cong \ast^n(V_{\chi_2}),$$  where $\ast^n$ denotes the
$n$-th convolution power. Let's calculate this convolution power.

Fix a nontrivial additive character $\psi:\mathbb Z/p\to
\overline{\mathbb Q}_\ell^\ast$. Let
$$g(\chi_2, \psi)=\sum_{x\in\mathbb F_p^\ast}\left(\frac{x}{p}\right)\psi(x)$$ be the
quadratic Gauss sum, and let $G(\chi_2,\psi)$ is be the one
dimensional $\overline{\mathbb Q}_\ell$-vector space with
$\mathrm{Gal}(\overline{\mathbb F}_p/\mathbb F_p)$ action so that
the geometric Frobenius element acts by scalar multiplication by
$-g(\chi_2,\psi)$. Let $V_{\epsilon^m}$ be the one dimensional
$\overline{\mathbb Q}_\ell$-vector space with
$\mathrm{Gal}(\overline{\mathbb F}_p/\mathbb F_p)$ action so that
the geometric Frobenius element acts by scalar multiplication by
$\left(\frac{-1}{p}\right)^m$. We can also regard $G(\chi_2,\psi)$
and $V_{\epsilon^m}$ as $\overline{\mathbb Q}_\ell$-vector spaces
with $\mathrm{Gal}(\overline K/K)$ action and with
$\mathrm{Gal}(\bar k/k)$ action through the canonical homomorphisms
$$\mathrm{Gal}(\overline K/K)\to\mathrm{Gal}(\bar k/k)\to
\mathrm{Gal}(\overline{\mathbb F}_p/\mathbb F_p).$$ We have
$$G(\chi_2,\psi)\otimes G(\chi_2, \psi)\cong V_{\epsilon}(-1).$$
This follows from the formula
$$(g(\chi_2,\psi))^2=\Big(\frac{-1}{p}\Big)p$$
for the quadratic Gauss sum. We thus have the following formula for
the $n$-th tensor power
$$\otimes^n (V_{\chi_2}\otimes
G(\chi_2,\psi))\cong \left\{\begin{array}{ll} V_{\epsilon^m}(-m)
&\hbox{if }n=2m \hbox { is
even},\\
V_{\chi_2}\otimes G(\chi_2,\psi)\otimes  V_{\epsilon^m}(-m)&\hbox{if
} n=2m+1 \hbox{ is odd}.
\end{array}\right.$$
Let $\mathscr F^{(0,\infty')}$ be Laumon's local Fourier
transformation \cite[2.4.2.3]{L} associated with the character
$\psi$. By \cite[2.5.3.1]{L}, we have
$$
\mathscr F^{(0,\infty')}\big(V_{\chi_2}\big)\cong V_{\chi_2}\otimes
G(\chi_2,\psi).$$ By \cite[2.7.2.2 (i)]{L}, we have
\begin{eqnarray*}
&&\mathscr F^{(0,\infty')}\Big(\ast^n\big(V_{\chi_2}\big)\Big)\\
&\cong& \otimes^n \mathscr F^{(0,\infty')}\big(V_{\chi_2}\big)\\
&\cong&\otimes^n(V_{\chi_2}\otimes
G(\chi_2,\psi))\\
&\cong& \left\{\begin{array}{ll} V_{\epsilon^m}(-m) &\hbox{if }n=2m
\hbox { is
even},\\
V_{\chi_2}\otimes G(\chi_2,\psi)\otimes  V_{\epsilon^m}(-m)&\hbox{if
} n=2m+1 \hbox{ is odd}.
\end{array}\right.
\end{eqnarray*}
On the other hand, the calculation of \cite[2.5.3.1]{L} shows that
\begin{eqnarray*}
\mathscr F^{(0,\infty')}\big(V_{\epsilon^m}(-m)\big)&\cong&
V_{\epsilon^m}(-m), \\
\mathscr F^{(0,\infty')}\big(V_{\chi_2}\otimes
V_{\epsilon^m}(-m)\big)&\cong& V_{\chi_2}\otimes
G(\chi_2,\psi)\otimes  V_{\epsilon^m}(-m).
\end{eqnarray*}
By the inversion formula for the local Fourier transformation
\cite[2.4.3 (i) c)]{L}, we must have
$$\ast^n\big(V_{\chi_2}\big)\cong \left\{\begin{array}{ll} V_{\epsilon^m}(-m)
&\hbox{if }n=2m \hbox { is
even},\\
V_{\chi_2}\otimes V_{\epsilon^m}(-m)&\hbox{if } n=2m+1 \hbox{ is
odd}.
\end{array}\right.$$ We thus get the following corollary:

\begin{corollary} Let $k$ be a perfect field of characteristic
$p\not=2$, let $f:\mathbb A_k^n\to\mathbb A_k^1$ be the $k$-morphism
defined by $$f(z)=z_1^2+\cdots +z_n^2,$$ and let
$R\Phi_f(\overline{\mathbb Q}_\ell)$ be the vanishing cycles complex
of the morphism $$\mathbb A_k^n\times_{\mathbb A_k^1}\mathbb
A^1_{(0)}\to \mathbb A^1_{(0)}$$ induced from $f$ by base change.
Then with the notation of 1.3, we have
$$\Big(R^{n-1}\Phi_f(\overline{\mathbb Q}_\ell)\Big)_0\cong
\left\{\begin{array}{ll} V_{\epsilon^m}(-m) &\hbox{if }n=2m \hbox {
is
even},\\
V_{\chi_2}\otimes V_{\epsilon^m}(-m)&\hbox{if } n=2m+1 \hbox{ is
odd}.
\end{array}\right.$$
\end{corollary}

Note that the above formulas are compatible with those of \cite[XV
2.2.5 D, E]{DK} for algebraically closed field $k$.

\section{Use of Laumon's local Fourier transformation}

The Artin-Schreier morphism
$$\mathbb A_k^1\to\mathbb A_k^1,\quad t\mapsto t^p-t$$ is a
$\mathbb Z/p$-torsor. Fix a nontrivial additive character
$\psi:\mathbb Z/p\to \overline{\mathbb Q}_\ell^\ast$.
Pushing-forward the Artin-Schreier torsor using $\psi^{-1}$, we get
a lisse sheaf $\mathscr L_\psi$ on $\mathbb A_k^1$. Denote the
inverse image of $\mathscr L_\psi$ under the morphism
$$\mathbb A_k^1\times_k \mathbb A_k^1\to \mathbb A_k^1, \quad (t,t')\to
tt'$$ by $\mathscr L_\psi(tt')$. Let $\mathbb A_k^1\times_k \mathbb
A_k^1\hookrightarrow \mathbb A_k^1\times_k\mathbb P_k^1$ be the open
immersion defined by the canonical open immersion $\mathbb
A_k^1=\mathbb P_k^1-\{\infty'\}\hookrightarrow \mathbb P_k^1$.
Denote by $\overline {{\mathscr L}_\psi(tt')}$ the sheaf on $\mathbb
A_k^1\times_k\mathbb P_k^1$ obtained from the sheaf $\mathscr
L_\psi(tt')$ on $\mathbb A_k^1\times_k \mathbb A_k^1$ by extension
by zero. To distinguish the two factors in $\mathbb A_k^1\times_k
\mathbb A_k^1$ and in $\mathbb A_k^1\times_k \mathbb P_k^1$, we
denote objects related to the second factor by symbols with the
superscript $'$. Denote by $\mathbb P^1_{(\infty')}$ the
henselization of $\mathbb P_k^1$ at $\infty'$, denote by
$\eta_{\infty'}$ its generic point, and denote the restriction of
$\overline {{\mathscr L}_\psi(tt')}$ to $\mathbb
A^1_{(0)}\times_k\mathbb P_{(\infty')}^1$ also by $\overline
{{\mathscr L}_\psi(tt')}$. Fix a uniformizer $\pi$ of $S$. We have a
$k$-morphism $S\to \mathbb A_k^1$ induced by the $k$-homomorphism
$$k[t]\to \Gamma(S,\mathscr O_S), \quad t\mapsto \pi.$$ It induces a
$k$-morphism $S\to \mathbb A^1_{(0)}$ which we denote also by $\pi$.
Denote by $\overline{{\mathscr L}_\psi(\pi t')}$ the inverse image
of $\overline{{\mathscr L}_\psi(tt')}$ under the morphism
$$S\times_k \mathbb P^1_{(\infty')}\stackrel{\pi\times
\mathrm{id}_{\mathbb P^1_{(\infty')}}} \to \mathbb A^1_{(0)}\times_k
\mathbb P^1_{(\infty')}.$$ Our proof of Theorem \ref{mainthm} relies
on the following lemma:

\begin{lemma}\label{keylemma} Let $g:Y\to S$ be a flat morphism
of finite type, let $K\in \mathrm{ob}\,D_c^b(Y,\overline{\mathbb
Q}_\ell)$, and let $y$ be a $k$-rational point in the special fiber
$g^{-1}(s)$. Suppose that $Y$ is regular and pure of dimension $n$,
$g|_{Y-\{y\}}$ is smooth, and the sheaves $\mathscr
H^q(K)|_{Y-\{y\}}$ are lisse for all $q$.

(i) $R\Phi_{\eta_{\infty'}}\Big(p_1^\ast
j_!(R\Phi_{g}(K)_{y})\otimes^L \overline{\mathscr L_\psi(\pi
t')}\Big)$ is supported at $(s,\overline \infty')$, and
$R\Phi_{\eta_{\infty'}}\Big({\mathrm{pr}}_1^\ast K\otimes^L (g\times
\mathrm{id}_{{\mathbb P}_{(\infty')}^1})^\ast \overline{\mathscr
L_\psi(\pi t')}\Big)$ is supported at $(y,\overline \infty')$, where
$R\Phi_{g}(K)_{y}$ is a complex of
$\mathrm{Gal}(\bar\eta/\eta)$-module and is regarded as an object in
$D_c^b(\eta,\overline{\mathbb Q}_\ell)$, $j:\eta\hookrightarrow S$
is the canonical open immersion, $R\Phi_{\eta_{\infty'}}$ denotes
the vanishing cycles functors for the projections
$$S\times_k \mathbb P^1_{(\infty')}\to \mathbb
P^1_{(\infty')}, \quad Y\times_k \mathbb P^1_{(\infty')}\to \mathbb
P^1_{(\infty')},$$ and $p_1, {\mathrm{pr}}_1$ are the projections
$$p_1:
S\times_k \mathbb P^1_{(\infty')}\to S,\quad
{\mathrm{pr}}_1:Y\times_k \mathbb P^1_{(\infty')}\to Y.$$

(ii) We have a canonical isomorphism
$$
R\Phi_{\eta_{\infty'}}\Big(p_1^\ast j_!(R\Phi_{g}(K)_{y})\otimes^L
\overline{\mathscr L_\psi(\pi t')}\Big)_{(s,\infty')}\cong
R\Phi_{\eta_{\infty'}}\Big({\mathrm{pr}}_1^\ast K\otimes^L (g\times
\mathrm{id}_{{\mathbb P}_{(\infty')}^1})^\ast \overline{\mathscr
L_\psi(\pi t')}\Big)_{(y,\infty')}.$$

(iii) Suppose $K$ is a lisse sheaf. Then
$R^i\Phi_{\eta_{\infty'}}\Big({\mathrm{pr}}_1^\ast K\otimes^L
(g\times \mathrm{id}_{{\mathbb P}_{(\infty')}^1})^\ast
\overline{\mathscr L_\psi(\pi t')}\Big)$ vanishes for $i\not=n$, and
$R^n\Phi_{\eta_{\infty'}}\Big({\mathrm{pr}}_1^\ast
K\otimes^L(g\times \mathrm{id}_{{\mathbb P}_{(\infty')}^1})^\ast
\overline{\mathscr L_\psi(\pi t')}\Big)$ is a skyscraper sheaf on
$Y\times_k{\overline \infty'}$ supported at $(y,\overline \infty')$.

(iv) Under the condition of (iii), we have a canonical isomorphism
of $\mathrm{Gal}(\bar\eta_{\infty'}/\eta_{\infty'})$-modules
$$\mathscr
F^{(0,\infty')}\Big(R^{n-1}\Phi_{g}(K)_{y}\Big)\cong
R^n\Phi_{\eta_{\infty'}}\Big({\mathrm{pr}}_1^\ast K\otimes^L
(g\times \mathrm{id}_{{\mathbb P}_{(\infty')}^1})^\ast
\overline{\mathscr L_\psi(\pi t')}\Big)_{(y,\infty')}.$$
\end{lemma}

We will prove Lemma \ref{keylemma} in \S 3. Let's deduce Theorem
\ref{mainthm} from Lemma \ref{keylemma} and the K\"unneth formula
for the nearby cycles functor in \cite[4.7]{I}.

\begin{proof}[Proof of Theorem \ref{mainthm}] Since $X_i$ $(i=1,2)$ are regular
and $k$ is perfect, $X_i$ are smooth over $k$, and hence
$X=X_1\times_k X_2$ is smooth over $k$. It follows that $X$ is
regular. Note that for any smooth $k$-scheme $Y$ and any
$k$-morphism of finite type $f:Y\to\mathbb A_k^1$, $f$ is smooth at
a $k$-point $y$ of $Y$ if and only if $df(y)\in \Omega^1_{Y/k,
y}\otimes_{\mathscr O_Y,y} k(y)$ is nonzero, where by abuse of
notation, we denote also by $f$ the image of $t$ under the
$k$-homomorphism
$$k[t]\to \Gamma(Y,\mathscr O_Y)$$ corresponding to the $k$-morphism
$f$. This follows from \cite[17.11.1 (b)$\Leftrightarrow({\mathrm
c}')$]{EGA}.

Let's prove $f|_{X-\{x\}}$ is smooth. By base change to an algebraic
closure of $k$, we are reduced to the case where $k$ is
algebraically closed. Suppose $x^{(0)}=(x_1^{(0)}, x_2^{(0)})$ is a
$k$-rational point in $X=X_1\times_k X_2$ where $f$ is not smooth.
By our assumption, $f$ is smooth on $(X_1-\{x_1\})\times
(X_2-\{x_2\})$. It follows that either $x_1^{(0)}=x_1$ or
$x_2^{(0)}=x_2$. Without loss of generality, assume $x_1^{(0)}=x_1$.
If $x_2^{(0)}\not=x_2$, then by our assumption, $f_2$ is smooth at
$x_2^{(0)}$. By the discussion above, $df_2$ defines a nonzero
element in $\Omega^1_{X_2/k,x_2^{(0)}}\otimes_{\mathscr
O_{X_2,x_2^{(0)}}}k(x_2^{(0)})$. Under the identification
$$\Omega^1_{X/k}\cong p_1^\ast\Omega^1_{X_1/k}\bigoplus
p_2^\ast\Omega^1_{X_2/k},$$ where $p_i:X_1\times_k X_2\to X_i$
$(i=1,2)$ are the projections, $df$ is identified with
$(df_1,df_2)$. It follows that $df$ defines a nonzero element in
$\Omega^1_{X/k,x^{(0)}}\otimes_{\mathscr O_{X,x^{(0)}}}k(x^{(0)})$.
So $f$ is smooth at $x^{(0)}$. Contradiction. So we must have
$x_2^{(0)}=x_2$ and hence $x^{(0)}=x$. This proves (i).

By \cite[Sommes trig. 1.3.1]{DSGA}, we have a canonical isomorphism
$$(f_1\times \mathrm{id}_{{\mathbb P}_k^1})^\ast
\overline{\mathscr L_\psi(tt')}\boxtimes^L (f_2\times
\mathrm{id}_{{\mathbb P}_k^1})^\ast \overline{\mathscr
L_\psi(tt')}\cong (f\times \mathrm{id}_{{\mathbb P}_k^1})^\ast
\overline{\mathscr L_\psi(tt')},$$ where the external tensor product
is taken with respect to the Cartesian diagram
$$\begin{array}{ccc}
X_1\times_k X_2\times_k\mathbb P_k^1 &\to& X_2\times_k \mathbb P_k^1\\
\downarrow&&\downarrow\\
X_1\times_k \mathbb P_k^1&\to& \mathbb P_k^1.
\end{array}$$
So we have
$$\Big({\mathrm{pr}}_{X_1}^\ast
K_1\otimes^L(f_1\times \mathrm{id}_{{\mathbb P}_k^1})^\ast
\overline{\mathscr
L_\psi(tt')}\Big)\boxtimes^L\Big({\mathrm{pr}}_{X_2}^\ast
K_2\otimes^L (f_2\times \mathrm{id}_{{\mathbb P}_k^1})^\ast
\overline{\mathscr L_\psi(tt')}\Big)\cong {\mathrm{pr}}_X^\ast
K\otimes^L (f\times \mathrm{id}_{{\mathbb P}_k^1})^\ast
\overline{\mathscr L_\psi(tt')},$$ where
$${\mathrm{pr}}_X:X_1\times_k X_2\times_k \mathbb P_k^1\to
X_1\times_k X_2, \quad {\mathrm{pr}}_{X_1}:X_1\times_k \mathbb
P_k^1\to X_1, \quad {\mathrm{pr}}_{X_2}:X_2\times_k \mathbb P_k^1\to
X_2$$ are the projections. Therefore, we deduce from \cite[4.7]{I} a
canonical isomorphism
\begin{eqnarray*}
&&R\Psi_{\eta_{\infty'}}\Big({\mathrm{pr}}_{X_1}^\ast K_1\otimes^L
(f_1\times \mathrm{id}_{{\mathbb P}_k^1})^\ast \overline{\mathscr
L_\psi(tt')}\Big)\boxtimes^L
R\Psi_{\eta_{\infty'}}\Big({\mathrm{pr}}_{X_2}^\ast K_2
\otimes^L(f_2\times \mathrm{id}_{{\mathbb P}_k^1})^\ast
\overline{\mathscr
L_\psi(tt')}\Big)\\
&\cong&R\Psi_{\eta_{\infty'}}\Big({\mathrm{pr}}_X^\ast K\otimes^L
(f\times \mathrm{id}_{{\mathbb P}_k^1})^\ast \overline{\mathscr
L_\psi(tt')}\Big).
\end{eqnarray*}
Since ${\mathrm{pr}}_X^\ast K\otimes^L(f\times \mathrm{id}_{{\mathbb
P}_k^1})^\ast \overline{\mathscr L_\psi(tt')}$ vanishes on
$X\times_k\overline\infty'$, we have
$$R\Phi_{\eta_{\infty'}}\Big({\mathrm{pr}}_X^\ast K\otimes^L(f\times \mathrm{id}_{{\mathbb
P}_k^1})^\ast \overline{\mathscr L_\psi(tt')}\Big)\cong
R\Psi_{\eta_{\infty'}}\Big({\mathrm{pr}}_X^\ast K\otimes^L (f\times
\mathrm{id}_{{\mathbb P}_k^1})^\ast \overline{\mathscr
L_\psi(tt')}\Big),$$ and we have similar isomorphisms if we replace
$f$ by $f_i$ and ${\mathrm{pr}}_X^\ast K$ by
${\mathrm{pr}}_{X_i}^\ast K_i$. So we have a canonical isomorphism
\begin{eqnarray}\label{iso(0)}
\begin{split}{}
&R\Phi_{\eta_{\infty'}}\Big({\mathrm{pr}}_{X_1}^\ast
K_1\otimes^L(f_1\times \mathrm{id}_{{\mathbb P}_k^1})^\ast
\overline{\mathscr L_\psi(tt')}\Big)\boxtimes^L
R\Phi_{\eta_{\infty'}}\Big({\mathrm{pr}}_{X_2}^\ast
K_2\otimes^L(f_2\times \mathrm{id}_{{\mathbb P}_k^1})^\ast
\overline{\mathscr
L_\psi(tt')}\Big)\\
\cong\;& R\Phi_{\eta_{\infty'}}\Big({\mathrm{pr}}_X^\ast
K\otimes^L(f\times \mathrm{id}_{{\mathbb P}_k^1})^\ast
\overline{\mathscr L_\psi(tt')}\Big).
\end{split}
\end{eqnarray}
Combined with Lemma \ref{keylemma} (i), we get
\begin{eqnarray*}
&&R\Phi_{\eta_{\infty'}}\Big({\mathrm{pr}}_{X_1}^\ast
K_1\otimes^L(f_1\times \mathrm{id}_{{\mathbb P}_k^1})^\ast
\overline{\mathscr L_\psi(tt')}\Big)_{(x_1,\infty')}\otimes^L
R\Phi_{\eta_{\infty'}}\Big({\mathrm{pr}}_{X_2}^\ast
K_2\otimes^L(f_2\times \mathrm{id}_{{\mathbb P}_k^1})^\ast
\overline{\mathscr
L_\psi(tt')}\Big)_{(x_2,\infty')}\\
&\cong& R\Phi_{\eta_{\infty'}}\Big({\mathrm{pr}}_X^\ast
K\otimes^L(f\times \mathrm{id}_{{\mathbb P}_k^1})^\ast
\overline{\mathscr L_\psi(tt')}\Big)_{(x,\infty')}.
\end{eqnarray*}
By Lemma \ref{keylemma} (ii), this induces a canonical isomorphism
\begin{equation}\label{iso(1)}
\begin{split}{}
&R\Phi_{\eta_{\infty'}}\Big(p_1^\ast
j_!(R\Phi_{f_1}(K_1)_{x_1})\otimes^L \overline{\mathscr
L_\psi(tt')}\Big)_{(0,\infty')}\otimes^L
R\Phi_{\eta_{\infty'}}\Big(p_1^\ast
j_!(R\Phi_{f_2}(K_2)_{x})\otimes^L \overline{\mathscr
L_\psi(tt')}\Big)_{(0,\infty')}\\
\cong\;&R\Phi_{\eta_{\infty'}}\Big(p_1^\ast
j_!(R\Phi_{f}(K)_{x})\otimes^L \overline{\mathscr
L_\psi(tt')}\Big)_{(0,\infty')},
\end{split}
\end{equation}
and under the assumption of Theorem \ref{mainthm} (iii), we get a
canonical isomorphism
\begin{eqnarray*}
&&\mathscr
F^{(0,\infty')}\Big(R^{n_1-1}\Phi_{f_1}(K_1)_{x_1}\Big)\otimes
\mathscr F^{(0,\infty')}\Big(R^{n_2-1}\Phi_{f_2}(K_2)_{x_2}\Big)\\
&\cong&\mathscr F^{(0,\infty')}\Big(R^{n-1}\Phi_{f}(K)_{x}\Big).
\end{eqnarray*}
By \cite[2.7.2.2 (i)]{L} and the inversion formula for local Fourier
transformation \cite[2.4.3 (i) c)]{L}, under the assumption of
Theorem \ref{mainthm} (iii), we have a canonical isomorphism
$$R^{n_1-1}\Phi_{f_1}(K_1)_{x_1}
\ast R^{n_2-1}\Phi_{f_2}(K_2)_{x_2}\cong R^{n-1}\Phi_{f}(K)_{x}.$$

The argument in the proof of various results in \cite{L} also shows
that in general we have a canonical isomorphism
$$R\Phi_{f_1}(K_1)_{x_1} \underline\ast R\Phi_{f_2}(K_2)_{x_2} \cong
R\Phi_{f}(K)_{x}.$$ We give a detailed argument for completeness.
Choose $L, L_1, L_2\in \mathrm{ob}\,D_c^b(\mathbb A_k^1-\{0\},
\overline{\mathbb Q}_\ell)$ such that $$L|_{\eta_0}\cong
R\Phi_f(K)_x, \quad L_1|_{\eta_0}\cong R\Phi_{f_1}(K_1)_{x_1}, \quad
L_2|_{\eta_0}\cong R\Phi_{f_2}(K_2)_{x_2},$$ and such that $\mathscr
H^q(L)$, $\mathscr H^q(L_1)$, $\mathscr H^q(L_2)$ are lisse on
$\mathbb A_k^1-\{0\}$ and tamely ramified at $\infty$ for all $q$.
The existence of $L, L_1, L_2$ follows from Lemma \ref{derivedlemma}
below. Let
$$\mathscr F:D_c^b(\mathbb A_k^1,\overline{\mathbb Q}_\ell)\to
D_c^b(\mathbb A_k^1,\overline{\mathbb Q}_\ell)$$ be the
Deligne-Fourier transformation (\cite[1.2.1.1]{L}), and let
$\iota:\mathbb A_k^1-\{0\}\hookrightarrow \mathbb A_k^1$ be the
canonical open immersion. By the same argument as the proof of
\cite[2.3.3.1]{L} (which relies on \cite[1.3.1.2]{L}), by the fact
that $\mathscr H^q(L)$ and $\mathscr H^q(L_i)$ are lisse on $\mathbb
A_k^1-\{0\}$ and tamely ramified at $\infty$,  and by \cite[2.4.3
(iii) b)]{L}, we have
\begin{eqnarray*}
\mathscr F(\iota_!L)[-1]|_{\eta_{\infty'}}&\cong&
R\Phi_{\eta_{\infty'}}\Big(p_1^\ast j_!(R\Phi_{f}(K)_{x})\otimes^L
\overline{\mathscr L_\psi(tt')}\Big)_{(0,\infty')},\\
\mathscr F(\iota_!L_1)[-1]|_{\eta_{\infty'}}&\cong&
R\Phi_{\eta_{\infty'}}\Big(p_1^\ast
j_!(R\Phi_{f_1}(K_1)_{x_1})\otimes^L \overline{\mathscr
L_\psi(tt')}\Big)_{(0,\infty')},\\
\mathscr
F(\iota_!L_2)[-1]|_{\eta_{\infty'}}&\cong&R\Phi_{\eta_{\infty'}}\Big(p_1^\ast
j_!(R\Phi_{f_2}(K_2)_{x_2})\otimes^L \overline{\mathscr
L_\psi(tt')}\Big)_{(0,\infty')}.
\end{eqnarray*}
So we can write the isomorphism (\ref{iso(1)}) as
$$(\mathscr F(\iota_!L_1)\otimes^L \mathscr F(\iota_!L_2))[-2]|_{\eta_{\infty'}}
\cong \mathscr F(\iota_!L)[-1]|_{\eta_{\infty'}}.$$ By
\cite[1.2.2.7]{L}, we get the isomorphism
$$\mathscr F(\iota_!L_1\ast \iota_!L_2)|_{\eta_{\infty'}}
\cong \mathscr F(\iota_!L)|_{\eta_{\infty'}},$$ where
$\iota_!L_1\ast \iota_!L_2$ is the (global) convolution product of
$\iota_!L_1$ and $\iota_!L_2$ (\cite[1.2.2.6]{L}). By the above
isomorphism and Lemma \ref{smalllemma} below, we have
$$R\Phi_{\eta_0}(\iota_!L_1\ast \iota_!L_2)\cong
R\Phi_{\eta_0}(\iota_!L),$$ where $R\Phi_{\eta_0}$ denotes the
vanishing cycles functor relative to the identity morphism on
$\mathbb A^1_{(0)}$. By \cite[2.7.1.1 (iii)]{L}, this last
isomorphism is exactly
$$R\Phi_{f_1}(K_1)_{x_1} \underline\ast
R\Phi_{f_2}(K_2)_{x_2} \cong R\Phi_{f}(K)_{x}.$$
\end{proof}

\begin{lemma}\label{derivedlemma}
For any object $K$ in $D_c^b(\eta_0,\overline {\mathbb Q}_\ell)$,
there exists an object $\overline K$ in $D_c^b(\mathbb
A_k^1-\{0\},\overline {\mathbb Q}_\ell)$ such that $\overline
K|_{\eta_0}\cong K$ and such that $\mathscr H^q(\overline K)$ are
lisse on $\mathbb A_k^1-\{0\}$ and tamely ramified at $\infty$ for
all $q$.
\end{lemma}

\begin{proof} By the description of the derived category of $\ell$-adic
sheaves in \cite[1.1]{DWeil}, there exists a finite extension $E$ of
$\mathbb Q_\ell$ such that $K$ is given by an object in
$D_c^b(\eta_0,R)$, where $R$ is the integer ring of $E$. Let
$\lambda$ be a uniformizer of $R$. The object $K$ corresponds to a
projective system defined by complexes $K_n\in\mathrm{ob} \,
D_{ctf}(\eta_0, R/(\lambda^n))$ and isomorphisms
$$K_{n+1}\otimes^L_{R/(\lambda^{n+1})}R/(\lambda^n)\cong K_n$$
in $D_{ctf}(\eta_0,R/(\lambda^n))$. By \cite[Rapport 4.6]{DSGA}, we
can represent each $K_n$ by a bounded complex so that all components
$K_n^i$ $(i\in\mathbb Z)$ are constructible flat sheaves of
$R/(\lambda^n)$-modules on $\eta_0$. As $\eta_0$ is the spectrum of
a field, a flat constructible sheaf of $R/(\lambda^n)$-modules is
just a free $R/(\lambda^n)$-module of finite rank with a continuous
$\mathrm{Gal}(\bar\eta_0/\eta_0)$-action. Denote $\mathbb
A_k^1-\{0\}$ by $\mathbb G_{m,k}$, and let $\pi_1(\mathbb G_{m,k},
\bar\eta_0)^{\mathrm{tame},\infty}$ be the quotient of
$\pi_1(\mathbb G_{m,k}, \bar\eta_0)$ classifying finite etale
coverings of $\mathbb G_{m,k}$ tamely ramified at $\infty$. By
\cite[2.2.2.2]{L}, the composite of the canonical homomorphisms
$$\mathrm{Gal}(\bar\eta_0/\eta_0)=\pi_1(\eta_0,\bar \eta_0)
\to \pi_1(\mathbb G_{m,k},\bar\eta_0)\twoheadrightarrow
\pi_1(\mathbb G_{m,k},\bar\eta_0)^{\mathrm{tame},\infty}$$ has a
retraction
$$r:\pi_1(\mathbb
G_{m,k},\bar\eta_0)^{\mathrm{tame},\infty}\to
\mathrm{Gal}(\bar\eta_0/\eta_0).$$ Through the composite
$$\pi_1(\mathbb G_{m,k},\bar\eta_0)\twoheadrightarrow
\pi_1(\mathbb
G_{m,k},\bar\eta_0)^{\mathrm{tame},\infty}\stackrel{r}\to
\mathrm{Gal}(\bar\eta_0/\eta_0),$$ each free $R/(\lambda^n)$-module
of finite rank with a continuous
$\mathrm{Gal}(\bar\eta_0/\eta_0)$-action can be endowed with a
continuous $\pi_1(\mathbb G_{m,k},\bar\eta_0)$-action extending the
given action of $\mathrm{Gal}(\bar\eta_0/\eta_0)$ at $0$ and tamely
ramified at $\infty$. In particular, each component $K_n^i$ of the
complex $K_n$ now becomes  a free $R/(\lambda^n)$-module of finite
rank with continuous $\pi_1(\mathbb G_{m,k},\bar\eta_0)$-action
extending the given action of $\mathrm{Gal}(\bar\eta_0/\eta_0)$ at
$0$ and tamely ramified at $\infty$. We thus get a lisse
constructible flat sheaf $\overline K_n^i$ of
$R/(\lambda^n)$-modules on $\mathbb G_{m,k}$ tamely ramified at
$\infty$ with the property $\overline K^i_n|_{\eta_0}=K_n^i$. The
differential morphisms $d^i_{K_n}:K_n^i\to K_n^{i+1}$ for the
complex $K_n$ are linear maps compatible with the
$\mathrm{Gal}(\bar\eta_0/\eta_0)$-action, and induce linear maps
compatible with the $\pi_1(\mathbb G_{m,k}.\bar\eta_0)$-action. So
they define morphisms of sheaves $d^i_{\overline K_n}:\overline
K_n^i\to \overline K_n^{i+1}$ with the property $d\circ d=0$. We
thus get a complex $\overline K_n\in\mathrm{ob}\,D_{ctf}(\mathbb
G_{m,k},R/(\lambda^n))$ with the property $\overline
K_n|_{\eta_0}\cong K_n$. Finally let's prove the isomorphisms
$$K_{n+1}\otimes^L_{R/(\lambda^{n+1})}R/(\lambda^n)\cong K_n$$
in $D_{ctf}(\eta_0,R/(\lambda^n))$ can be extended to isomorphisms
$$\overline K_{n+1}\otimes^L_{R/(\lambda^{n+1})}R/(\lambda^n)\cong \overline K_n$$
in $D_{ctf}(\mathbb G_{m,k},R/(\lambda^n))$. The projective system
$(\overline K_n)\in\mathrm{ob}\, D_c^b(\mathbb G_{m,k},R)$ then
defines an object $\overline K$ in $\mathrm{ob}\, D_c^b(\mathbb
G_{m,k},\overline{\mathbb Q}_\ell)$ with the required property.
Since components of $K_{n+1}$ are free $R/(\lambda^{n+1})$-modules
of finite rank, we have
$$ K_{n+1}\otimes^L_{R/(\lambda^{n+1})}R/(\lambda^n)\cong
K_{n+1}\otimes_{R/(\lambda^{n+1})}R/(\lambda^n).$$ The isomorphism
$K_{n+1}\otimes^L_{R/(\lambda^{n+1})}R/(\lambda^n)\cong K_n$ in
$D_{ctf}(\eta_0,R/(\lambda^n))$ can be represented by a diagram
$$\begin{array}{rlc}
L&&\\
\downarrow &\searrow&\\
K_{n+1}\otimes_{R/(\lambda^{n+1})}R/(\lambda^n)&& K_n,
\end{array}$$ where the two arrows are quasi-isomorphisms. By \cite[Rapport 4.7]{DSGA},
we may assume $L$ is a bounded above complex of free
$R/(\lambda^n)$-modules of finite rank with continuous
$\mathrm{Gal}(\bar\eta_0/\eta_0)$-action. Again through the
composite
$$\pi_1(\mathbb G_{m,k},\bar\eta_0)\twoheadrightarrow
\pi_1(\mathbb
G_{m,k},\bar\eta_0)^{\mathrm{tame},\infty}\stackrel{r}\to
\mathrm{Gal}(\bar\eta_0/\eta_0),$$ we can extend $L$ to a complex
$\overline L$ of lisse constructible flat sheaves of
$R/(\lambda^n)$-modules on $\mathbb G_{m,k}$ and extend the above
diagram to a diagram
$$\begin{array}{rlc}
\overline L&&\\
\downarrow &\searrow&\\
\overline K_{n+1}\otimes_{R/(\lambda^{n+1})}R/(\lambda^n)&&
\overline K_n \end{array}$$ of complexes of lisse sheaves of
$\mathbb G_{m,k}$ so that the two arrows are quasi-isomorphisms.
This defines the isomorphism
$$\overline K_{n+1}\otimes_{R/(\lambda^{n+1})}R/(\lambda^n)\cong \overline
K_n$$ in $D_{ctf}(\mathbb G_{m,k},R/(\lambda^n))$ that we are
seeking.
\end{proof}

\begin{lemma}\label{smalllemma} Let $L_1,L_2\in\mathrm{ob}\,D_c^b(\mathbb
A_k^1,\overline{\mathbb Q}_\ell)$. If $\mathscr F
(L_1)|_{\eta_{\infty'}}\cong \mathscr F (L_2)|_{\eta_{\infty'}}$,
then $R\Phi_{\eta_0}(L_1)\cong R\Phi_{\eta_0}(L_2)$.
\end{lemma}

\begin{proof}Let $L\in\mathrm{ob}\,D_c^b(\mathbb A_k^1,\overline{\mathbb
Q}_\ell)$ and let $L'=\mathscr F(L)$. By the inversion formula for
the Deligne-Fourier transformation \cite[1.2.2.1]{L}, we have
$$\mathscr F'(L')=b_\ast L(-1),$$ where $b:\mathbb
A_k^1\to\mathbb A_k^1$ is the morphism defined by $t\mapsto -t$, and
$\mathscr F'$ denote the Deligne-Fourier transformation for the dual
of $\mathbb A_k^1$ (which is the second factor of $\mathbb
A_k^1\times_k\mathbb A_k^1$). So we have
$$R\Phi_{\eta_0}(L)=b^\ast R\Phi_{\eta_0}(\mathscr F'(L'))(1).$$
By \cite[2.3.2.1 (i)]{L} (which relies on \cite[Th. finitude
2.16]{DSGA}), $R\Phi_{\eta_0}(\mathscr F'(L'))$ depends only on
$L'|_{\eta_{\infty'}}=\mathscr F(L)|_{\eta_{\infty'}}$. Our
assertion follows.
\end{proof}

\begin{remark} We expect that Theorem 1.2 still holds if we just assume $\mathscr H^q(K_i)|_{X_i-\{x_i\}}$
are lisse for all $q$. The difficulty is that we don't know whether
this weaker condition implies that $R\Phi_f(K)$ is supported at $x$.
Using Lemma 2.1 (i) and the canonical isomorphism (\ref{iso(0)}),
one can show that $R\Phi_{\eta_{\infty'}}\Big({\mathrm{pr}}_X^\ast
K\otimes^L(f\times \mathrm{id}_{{\mathbb P}_k^1})^\ast
\overline{\mathscr L_\psi(tt')}\Big)$ is supported at $(x,\infty')$
under this condition.
\end{remark}

\section{Proof of Lemma \ref{keylemma}}

We first prove Lemma \ref{keylemma} (i) and (iii).

\begin{proof}[Proof of Lemma \ref{keylemma} (i) and (iii)] That
$R\Phi_{\eta_{\infty'}}\Big(p_1^\ast j_!(R\Phi_{g}(K)_{y})\otimes^L
\overline{\mathscr L_\psi(\pi t')}\Big)$ is supported at
$(s,\overline \infty')$ follows from \cite[2.4.2.2]{L} and the fact
that the functor $R\Phi_{\eta_{\infty'}}\Big(p_1^\ast
j_!(-)\otimes^L \overline{\mathscr L_\psi(\pi t')}\Big)$ is exact.
Here $1/t'$ is a uniformizing parameter $\pi'$ on $\mathbb
P^1_{(\infty')}$, and in Laumon's notation, $\overline{\mathscr
L_\psi(\pi t')}$ is $\overline {\mathscr L}_\psi(\pi/\pi')$. By
\cite[1.3.1.2]{L}, $R\Phi_{\eta_{\infty'}}(\overline{{\mathscr
L}_\psi(\pi t')})$ vanishes on $S\times_k\overline\infty'$. On the
other hand, $g\times\mathrm{id}_{{\mathbb P}_{(\infty')}^1}$ is
smooth on $(Y-\{y\})\times_k\mathbb P_{(\infty')}^1$, and $\mathscr
H^q(K)|_{Y-\{y\}}$ are lisse for all $q$. By the smooth base change
theorem and the projection formula, we have
\begin{eqnarray*}
&&R\Phi_{\eta_{\infty'}}\Big({\mathrm{pr}}_1^\ast K\otimes^L(g\times
\mathrm{id}_{{\mathbb P}_{(\infty')}^1})^\ast \overline{{\mathscr
L}_\psi(\pi t')}\Big)|_{(Y-\{y\})\times_k\overline \infty'}\\
&\cong& ({\mathrm{pr}}_1^\ast K)|_{(Y-\{y\})\times_k\overline
\infty'} \otimes^L \Big((g\times
\mathrm{id}_{\overline\infty'})^\ast
R\Phi_{\eta_{\infty'}}(\overline{{\mathscr
L}_\psi(\pi t')})\Big)|_{(Y-\{y\})\times_k\overline \infty'}\\
&=&0.
\end{eqnarray*}
It follows that $R\Phi_{\eta_{\infty'}}\Big({\mathrm{pr}}_1^\ast
K\otimes^L (g\times \mathrm{id}_{{\mathbb P}_{(\infty')}^1})^\ast
\overline{\mathscr L_\psi(\pi t')}\Big)$ is supported at
$(y,\overline \infty')$. This proves Lemma \ref{keylemma} (i).

Note that the restriction of $\overline{{\mathscr L}_\psi(\pi t')}$
to $S\times_k \eta_{\infty'}$ is a lisse sheaf. So under the
assumption of (iii), the restriction of ${\mathrm{pr}}_1^\ast
K\otimes^L (g\times \mathrm{id}_{{\mathbb P}_{(\infty')}^1})^\ast
\overline{{\mathscr L}_\psi(\pi t')}[n]$ to
$Y\times_k\eta_{\infty'}$ is perverse. By \cite[4.5]{I},
$R\Psi_{\eta_{\infty'}}\Big({\mathrm{pr}}_1^\ast K\otimes^L (g\times
\mathrm{id}_{{\mathbb P}_{(\infty')}^1})^\ast \overline{{\mathscr
L}_\psi(\pi t')}\Big)[n]$ is perverse on
$Y\times_k\overline\infty'$. Note that ${\mathrm{pr}}_1^\ast
K\otimes^L(g\times \mathrm{id}_{{\mathbb P}_{(\infty')}^1})^\ast
\overline{{\mathscr L}_\psi(\pi t')}$ vanishes on
$Y\times_k\overline\infty'$. So we have
$$R\Phi_{\eta_{\infty'}}\Big({\mathrm{pr}}_1^\ast
K\otimes^L(g\times \mathrm{id}_{{\mathbb P}_{(\infty')}^1})^\ast
\overline{{\mathscr L}_\psi(\pi t')}\Big)\cong
R\Psi_{\eta_{\infty'}}\Big({\mathrm{pr}}_1^\ast K\otimes^L(g\times
\mathrm{id}_{{\mathbb P}_{(\infty')}^1})^\ast \overline{{\mathscr
L}_\psi(\pi t')}\Big).$$ Hence
$R\Phi_{\eta_{\infty'}}\Big({\mathrm{pr}}_1^\ast K\otimes^L(g\times
\mathrm{id}_{{\mathbb P}_{(\infty')}^1})^\ast \overline{{\mathscr
L}_\psi(\pi t')}\Big)[n]$ is a perverse sheaf on
$Y\times_k\overline\infty'$ supported at $(y,\overline\infty')$.
Lemma \ref{keylemma} (iii) then follows.
\end{proof}

\begin{lemma}\label{techlemma} Let $K$ be an object in $D_c^b(S,\overline
{\mathbb Q}_\ell)$, let
$$p_1:S\times_k \mathbb P^1_{(\infty')}\to S,
\quad p_2:S\times_k \mathbb P^1_{(\infty')}\to \mathbb
P^1_{(\infty')}$$ be the projections, and let $j:\eta\hookrightarrow
S$ be the canonical open immersion. We have a canonical isomorphism
$$R\Phi_{\eta_{\infty'}}\Big(p_1^\ast K\otimes^L \overline{\mathscr
L_\psi(\pi t')}\Big)\cong R\Phi_{\eta_{\infty'}}\Big(p_1^\ast
\big(j_!R\Phi_{\eta}(K)\big)\otimes^L \overline{\mathscr L_\psi(\pi
t')}\Big),$$ where $R\Phi_{\eta_{\infty'}}$ denotes the vanishing
cycles functor with respect to the projection $p_2$,
$R\Phi_{\eta}(K)$ denotes the vanishing cycles complex of $K$
relative to $\mathrm{id}_S$, which is a complex of
$\mathrm{Gal}(\bar\eta/\eta)$-modules and is regarded as an object
in $D_c^b(\eta,\overline {\mathbb Q}_\ell)$, and can be identified
with the cone of the specialization morphism $K_{\bar s}\to K_{\bar
\eta}$.
\end{lemma}

\begin{proof} Let $E$ be a finite extension of $\mathbb Q_\ell$
in $\overline{\mathbb Q}_\ell$ containing $\mathrm{im}(\psi)$, let
$R$ be the integral closure of $\mathbb Z_\ell$ in $E$, and let
$\lambda$ be a uniformizer of $R$. It suffices to prove the same
assertion for any complex $K$ of sheaves of $R/(\lambda^m)$-modules
for any positive integer $m$. For convenience, write $\Lambda$ for
$R/(\lambda^m)$. For any scheme $X$, denote by $C(X, \Lambda)$ the
triangulated category of complexes of etale sheaves of
$\Lambda$-modules on $X$ with morphisms being homotopy classes of
morphisms of complexes. Let $\tilde S$ be the strict henselization
of $S$ at a geometric point $\bar s$ over $s$. Fix notation by the
following commutative diagram:
$$\begin{array}{ccccc}
\bar \eta&\stackrel{\bar j}\rightarrow & \tilde S
&\stackrel{\bar i}\leftarrow &\bar s\\
\downarrow &&\downarrow && \downarrow \\
 \eta&\stackrel{j}\rightarrow & S&\stackrel{i}\leftarrow &s.
\end{array}$$
For any complex $K\in \mathrm{ob}\, C(S, \Lambda)$, we have
\begin{eqnarray*}
R\Gamma(\bar s, R\Psi_{\eta}(K))&=&R\Gamma(\bar s,\bar i^\ast R\bar
j_\ast
\bar j^\ast (K|_{\tilde S}))\\
&\cong&\Gamma(\bar s, \bar i^\ast \bar j_\ast \bar j^\ast
(K|_{\tilde S}))\\
&\cong& K_{\bar \eta}.
\end{eqnarray*}
So $R\Psi_{\eta}(K)$ corresponds to the complex $K_{\bar\eta}$ of
$\Lambda$-modules with $\mathrm{Gal}(\bar\eta/\eta)$-action. Recall
that $R\Phi_{\eta}(K)$ is the mapping cone of the canonical morphism
$\bar i^\ast (K|_{\tilde S})\to R\Psi_{\eta}(K)$.

Let $I$ be the inertia subgroup of $\mathrm{Gal}(\bar\eta/\eta)$,
that is, the kernel of the canonical epimorphism
$\mathrm{Gal}(\bar\eta/\eta)\to \mathrm{Gal}(\bar s/s)$. The functor
$$K\mapsto (i^\ast K,
j^\ast K, i^\ast K\to i^\ast j_\ast j^\ast K)$$ defines an
equivalence of categories from the category $C(S,\Lambda)$ to the
category of triples
$$(F, G, F\to G^I),$$ where $F$ is
a complex of $\Lambda$-modules with $\mathrm{Gal}(\bar s/s)$-action
which can be identified with  an object in $C(s,\Lambda)$, $G$ is a
complex of $\Lambda$-modules with
$\mathrm{Gal}(\bar\eta/\eta)$-action which can be identified with an
object in $C(\eta,\Lambda)$, and $F\to G^I$ is an equivariant
morphism complexes. Given an object $K$ in $C(S,\Lambda)$, consider
the triples
$$K'=(K_{\bar s}, K_{\bar s}, K_{\bar s}\stackrel{\mathrm{id}}\to
K_{\bar s}),\quad K=(K_{\bar s}, K_{\bar\eta}, K_{\bar s}\to
K_{\bar\eta}^I).$$ Note that the second object is exactly the triple
associated to $K$ and hence is denoted by $K$. The first object
corresponds to a complex of constant sheaves on $S$. We have a
canonical morphism $K'\to K$. Let $K''$ be the mapping cone of
$K'\to K$. Then the canonical morphism $j_!j^\ast K''\to K''$
defines a quasi-isomorphism $j_!R\Phi_{\eta}(K)\to K''$. We thus
have a distinguished triangle
$$K'\to K\to j_!R\Phi_{\eta}(K)\to$$ in the derived category
$D(S, \Lambda)$. It gives rise to a distinguished triangle
$$R\Psi_{\eta_{\infty'}}\Big(p_1^\ast K'\otimes^L \overline{\mathscr
L_\psi(\pi t')}\Big)\to R\Psi_{\eta_{\infty'}}\Big(p_1^\ast
K\otimes^L \overline{\mathscr L_\psi(\pi t')}\Big)\to
R\Psi_{\eta_{\infty'}}\Big(p_1^\ast
\big(j_!R\Phi_{\eta}(K)\big)\otimes^L \overline{\mathscr L_\psi(\pi
t')}\Big)\to,$$ which can be identified with a distinguished
triangle
$$R\Phi_{\eta_{\infty'}}\Big(p_1^\ast K'\otimes^L \overline{\mathscr
L_\psi(\pi t')}\Big)\to R\Phi_{\eta_{\infty'}}\Big(p_1^\ast
K\otimes^L \overline{\mathscr L_\psi(\pi t')}\Big)\to
R\Phi_{\eta_{\infty'}}\Big(p_1^\ast
\big(j_!R\Phi_{\eta}(K)\big)\otimes^L \overline{\mathscr L_\psi(\pi
t')}\Big)\to$$ because $\overline{\mathscr L_\psi(\pi t')}$ vanishes
on $S\times_k\overline \infty'$. Since $K'$ corresponds to a complex
of constant sheaves on $S$, we have
$$R\Phi_{\eta_{\infty'}}\Big(p_1^\ast K'\otimes^L \overline{\mathscr
L_\psi(\pi t')}\Big)=0$$ by \cite[1.3.1.2]{L}. (Note that
\cite[1.3.1.2]{L} holds also for the torsion coefficients $\Lambda$
as its proof in \cite[Appendice 2.4]{KL} shows.) We thus have
$$R\Phi_{\eta_{\infty'}}\Big(p_1^\ast K\otimes^L
\overline{\mathscr L_\psi(\pi t')}\Big)\cong
R\Phi_{\eta_{\infty'}}\Big(p_1^\ast
\big(j_!R\Phi_{\eta}(K)\big)\otimes^L \overline{\mathscr L_\psi(\pi
t')}\Big).$$
\end{proof}

\begin{lemma}\label{canonical} Let $f:
X\to S$ a morphism, $x$ a $k$-rational point in the special fiber
$f^{-1}(s)$, $X_{(x)}$ the henselization of $X$ at $x$,
$f_{(x)}:X_{(x)}\to S$ the morphism induced by $f$, and
$K\in\mathrm{ob}\, D_c^b(X,\overline{\mathbb Q}_\ell)$. Then we have
a canonical isomorphism
$$R\Phi(Rf_{(x)\ast} (K|_{X_{(x)}}))\cong (R\Phi(K))_{x}.$$
\end{lemma}

\begin{proof} Let $X_{(\bar x)}$ (resp. $\tilde S$)
be the strict henselization of $X$ (resp. $S$) at a geometric point
$\bar x$ (resp. $\bar s$) over $x$ (resp. $s$). Since $x$ is a
$k$-rational point, we have $X_{(\bar x)}\cong X_{(x)}\times_S
\tilde S$. Let $f_{(\bar x)}:X_{(\bar x)}\to \tilde S$ be the
morphism induced by $f$. It can be identified with the base change
of $f_{(x)}:X_{(x)}\to S$. Fix notation by the following commutative
diagram:
$$\begin{array}{rclcl}
X_{(\bar x)}\times_{\tilde S}\bar\eta &\rightarrow &X_{(\bar
x)}&\leftarrow &X_{(\bar x)}\times_{\tilde S}\bar s \\
{\scriptstyle f_{(\bar x),\bar
\eta}}\downarrow&&\downarrow{\scriptstyle f_{(\bar x)}}
&&  \downarrow{\scriptstyle f_{(\bar x),\bar s}}\\
\bar\eta &\stackrel{\bar j}\rightarrow &\tilde S&\stackrel{\bar
i}\leftarrow &\bar s.
\end{array}
$$
For convenience, denote the restrictions of $K$ to $X_{(x)}$,
$X_{(\bar x)}$ and $X_{(\bar x)}\times_{\tilde S}\bar\eta$ also by
$K$. Then $R\Phi(Rf_{(x)\ast} K)$ is the mapping cone of the
canonical morphism
$$\bar i^\ast Rf_{(\bar x)\ast}K\to \bar i^\ast R\bar j_\ast
\bar j^\ast Rf_{(\bar
x)\ast}K.$$ We have
\begin{eqnarray*}
R\Gamma(\bar s,\bar i^\ast Rf_{(\bar x)\ast}K)&\cong& R\Gamma(\tilde
S,Rf_{(\bar x)\ast}K)\\
&\cong& R\Gamma(X_{(\bar x)},K),\\
R\Gamma(\bar s,\bar i^\ast R\bar j_\ast \bar j^\ast Rf_{(\bar
x)\ast}K) &\cong&
R\Gamma(\tilde S,R\bar j_\ast \bar j^\ast Rf_{(\bar x)\ast}K)\\
&\cong& R\Gamma(\bar \eta,\bar j^\ast Rf_{(\bar x)\ast}K)\\
&\cong& R\Gamma(\bar \eta, Rf_{(\bar x),\bar\eta \ast}K)\\
&\cong&R\Gamma(X_{(\bar x)}\times_{\tilde S}\bar\eta, K).
\end{eqnarray*}
It follows that $R\Phi(Rf_{(x)\ast} K)$ can be identified with the
mapping cone of the canonical morphism
$$R\Gamma(X_{(\bar x)},K)\to R\Gamma(X_{(\bar x)}\times_{\tilde S}\bar\eta, K).$$
The later is exactly $(R\Phi(K))_{x}$.
\end{proof}

Under the assumption of Lemma \ref{keylemma}, let $Y_{(y)}$ be the
henselization of $Y$ at $y$, and let $g_{(y)}:Y_{(y)}\to S$ be the
morphism induced by $g$. Fix notation by the following commutative
diagram:
$$\begin{array}{ccccc}
Y_{(y)}\times_k \mathbb
P_{(\infty')}^1&\stackrel{g_{(y)}\times\mathrm{id}_{\mathbb
P_{(\infty')}^1}}\to&S\times_k\mathbb P_{(\infty')}^1&\stackrel
{p_2}\to&
\mathbb P_{(\infty')}^1\\
{\scriptstyle {\mathrm{pr}}_{1(y)}}\downarrow&&{\scriptstyle p_1}\downarrow&&\downarrow\\
Y_{(y)}&\stackrel{g_{(y)}}\to&S&\to& \mathrm{Spec}\,k.
\end{array}
$$ Denote the restriction of $K$ to $Y_{(y)}$ also by $K$.
We have canonical morphisms
\begin{eqnarray*}
&&R\Psi_{\eta_{\infty'}}\Big(p_1^\ast Rg_{(y)\ast}K
 \otimes^L \overline{\mathscr L_\psi(\pi t')}\Big)\\
&\cong&
R\Psi_{\eta_{\infty'}}\Big(R(g_{(y)}\times\mathrm{id}_{\mathbb
P^1_{(\infty')}})_\ast {\mathrm{pr}}_{1(y)}^\ast K\otimes^L
\overline{\mathscr L_\psi(\pi t')}\Big)\quad(\hbox{smooth base change theorem})\\
&\cong&
R\Psi_{\eta_{\infty'}}\Big(R(g_{(y)}\times\mathrm{id}_{\mathbb
P^1_{(\infty')}})_\ast\big({\mathrm{pr}}_{1(y)}^\ast K\otimes^L
(g_{(y)}\times\mathrm{id}_{\mathbb P^1_{(\infty')}})^\ast
\overline{\mathscr L_\psi(\pi t')}\big)\Big)\quad(\hbox{projection
formula})\\
&\to & R(g_{(y)}\times\mathrm{id}_{\overline \infty'})_\ast
R\Psi_{\eta_{\infty'}}\Big({\mathrm{pr}}_{1(y)}^\ast
K\otimes^L(g_{(y)}\times \mathrm{id}_{{\mathbb
P}_{(\infty')}^1})^\ast \overline{\mathscr L_\psi(\pi t')}\Big)
\end{eqnarray*}
Since $p_1^\ast Rg_{(y)\ast}K \otimes^L \overline{\mathscr
L_\psi(\pi t')}$ and ${\mathrm{pr}}_{1(y)}^\ast K\otimes^L
(g_{(y)}\times \mathrm{id}_{{\mathbb P}_{(\infty')}^1})^\ast
\overline{\mathscr L_\psi(\pi t')}$ vanish on the fibers over
$\infty'$, their vanishing cycles complex and nearby cycles complex
coincide. So the composite of the above canonical morphisms can be
identified with a canonical morphism
$$R\Phi_{\eta_{\infty'}}\Big(p_1^\ast Rg_{(y)\ast}K
\otimes^L \overline{\mathscr L_\psi(\pi t')}\Big)\to
R(g_{(y)}\times\mathrm{id}_{\overline\infty'})_\ast
R\Phi_{\eta_{\infty'}}\Big({\mathrm{pr}}_{1(y)}^\ast K\otimes^L
(g_{(y)}\times \mathrm{id}_{{\mathbb P}_{(\infty')}^1})^\ast
\overline{\mathscr L_\psi(\pi t')}\Big).$$ Applying Lemma
\ref{techlemma} to the complex $Rg_{(y)\ast}K$ on $S$, we get a
canonical isomorphism
$$R\Phi_{\eta_{\infty'}}\Big(p_1^\ast Rg_{(y)\ast}K
\otimes^L \overline{\mathscr L_\psi(\pi t')}\Big)\cong
R\Phi_{\eta_{\infty'}}\Big(p_1^\ast
\big(j_!R\Phi(Rg_{(y)\ast}K)\big)\otimes^L \overline{\mathscr
L_\psi(\pi t')}\Big).$$ By Lemma \ref{canonical}, we have
$$R\Phi(Rg_{(y)\ast}K)\cong (R\Phi_g(K))_{y}.$$ We thus get
a canonical morphism
$$R\Phi_{\eta_{\infty'}}\Big(p_1^\ast
\big(j_!(R\Phi_g(K))_{y}\big)\otimes^L \overline{\mathscr L_\psi(\pi
t')}\Big) \to R(g_{(y)}\times\mathrm{id}_{\overline\infty'})_\ast
R\Phi_{\eta_{\infty'}}\Big({\mathrm{pr}}_{1(y)}^\ast
K\otimes^L(g_{(y)}\times \mathrm{id}_{{\mathbb
P}_{(\infty')}^1})^\ast \overline{\mathscr L_\psi(\pi t')}\Big).$$
By Lemma \ref{keylemma} (i) that we have shown at the beginning of
this section, this gives rise to a canonical morphism
\begin{eqnarray}\label{iso(2)}
\begin{split}& R\Phi_{\eta_{\infty'}}\Big(p_1^\ast
\big(j_!(R\Phi_g(K))_{y}\big)\otimes^L \overline{\mathscr L_\psi(\pi
t')}\Big)_{(s,\infty')} \\
\to\;& R\Phi_{\eta_{\infty'}}\Big({\mathrm{pr}}_1^\ast
K\otimes^L(g\times \mathrm{id}_{{\mathbb P}_{(\infty')}^1})^\ast
\overline{\mathscr L_\psi(\pi t')}\Big)_{(y,\infty)}.
\end{split}
\end{eqnarray}
Under the assumption of Lemma \ref{keylemma} (iii), taking the
$n$-th cohomology on both sides and using the definition of
$\mathscr F^{(0,\infty')}$ in \cite[2.4.2.3]{L}, we get a canonical
morphism
$$\mathscr F^{(0,\infty')}\Big(R^{n-1}\Phi_{g}(K)_{y}\Big)\to
R^n\Phi_{\eta_{\infty'}}\Big({\mathrm{pr}}_1^\ast K\otimes^L
(g\times \mathrm{id}_{{\mathbb P}_{(\infty')}^1})^\ast
\overline{\mathscr L_\psi(\pi t')}\Big)_{(y,\infty')}.$$

\begin{lemma}\label{globallemma}  Let $f:X\to S$ be a proper
morphism, and let $K\in\mathrm{ob}\, D_c^b(X,\overline{\mathbb
Q}_\ell)$. Suppose $X$ is regular pure of dimension $n$, $f$ is
smooth at points in $f^{-1}(s)-A$ for a finite set $A$ of
$k$-rational points in $f^{-1}(s)$ and the sheaves $\mathscr
H^q(K)|_{X-A}$ are lisse for all $q$. Then for all $x\in A$, the
canonical morphisms
$$R\Phi_{\eta_{\infty'}}\Big(p_1^\ast
\big(j_!(R\Phi_f(K))_{x}\big)\otimes^L \overline{\mathscr L_\psi(\pi
t')}\Big)_{(s,\infty')} \to
R\Phi_{\eta_{\infty'}}\Big({\mathrm{pr}}_1^\ast K\otimes^L(f\times
\mathrm{id}_{{\mathbb P}_{(\infty')}^1})^\ast \overline{\mathscr
L_\psi(\pi t')}\Big)_{(x,\infty)}$$ constructed above are
isomorphisms, where ${\mathrm{pr}}_1:X\times_k \mathbb
P^1_{(\infty')}\to X$ is the projection. Suppose furthermore that
$K$ is a lisse sheaf. Then the canonical morphisms
$$\mathscr F^{(0,\infty')}\Big(R^{n-1}\Phi_{f}(K)_{x}\Big)\to
R^n\Phi_{\eta_{\infty'}}\Big({\mathrm{pr}}_1^\ast K\otimes^L
(f\times \mathrm{id}_{{\mathbb P}_{(\infty')}^1})^\ast
\overline{\mathscr L_\psi(\pi t')}\Big)_{(x,\infty')}$$ are
isomorphisms for all $x\in A$.
\end{lemma}

\begin{proof} The second statement follows from the first one. To prove the
first statement, it suffices to prove the above canonical morphisms
induce an isomorphism
$$\bigoplus_{x\in A}R\Phi_{\eta_{\infty'}}\Big(p_1^\ast
\big(j_!(R\Phi_f(K))_{x}\big)\otimes^L \overline{\mathscr L_\psi(\pi
t')}\Big)_{(s,\infty')} \to \bigoplus_{x\in A}
R\Phi_{\eta_{\infty'}}\Big({\mathrm{pr}}_1^\ast K\otimes^L(f\times
\mathrm{id}_{{\mathbb P}_{(\infty')}^1})^\ast \overline{\mathscr
L_\psi(\pi t')}\Big)_{(x,\infty)}.$$ Fix notation by the following
commutative diagram
$$\begin{array}{ccccc}
X\times_k \mathbb
P_{(\infty')}^1&\stackrel{f\times\mathrm{id}_{\mathbb
P_{(\infty')}^1}}\to&S\times_k\mathbb P_{(\infty')}^1&\stackrel
{p_2}\to&
\mathbb P_{(\infty')}^1\\
{\scriptstyle {\mathrm{pr}}_1}\downarrow&&{\scriptstyle p_1}\downarrow&&\downarrow\\
X&\stackrel{f}\to&S&\to& \mathrm{Spec}\,k.
\end{array}
$$
We have canonical isomorphisms
\begin{eqnarray*}
&&R\Psi_{\eta_{\infty'}}\Big(p_1^\ast Rf_{\ast}K
 \otimes^L \overline{\mathscr L_\psi(\pi t')}\Big)\\
&\cong& R\Psi_{\eta_{\infty'}}\Big(R(f\times\mathrm{id}_{\mathbb
P^1_{(\infty')}})_\ast {\mathrm{pr}}_{1}^\ast K\otimes^L
\overline{\mathscr L_\psi(\pi t')}\Big)\quad(\hbox{smooth base change theorem})\\
&\cong& R\Psi_{\eta_{\infty'}}\Big(R(f\times\mathrm{id}_{\mathbb
P^1_{(\infty')}})_\ast\big({\mathrm{pr}}_{1}^\ast K\otimes^L
(f\times\mathrm{id}_{\mathbb P^1_{(\infty')}})^\ast
\overline{\mathscr L_\psi(\pi t')}\big)\Big)\quad(\hbox{projection
formula})\\
&\cong & R(f\times\mathrm{id}_{\overline \infty'})_\ast
R\Psi_{\eta_{\infty'}}\Big({\mathrm{pr}}_1^\ast K\otimes^L(f\times
\mathrm{id}_{{\mathbb P}_{(\infty')}^1})^\ast \overline{\mathscr
L_\psi(\pi t')}\Big)\quad(\hbox{proper base change theorem})
\end{eqnarray*}
Since $p_1^\ast Rf_\ast K\otimes^L\overline{\mathscr L_\psi(\pi
t')}$ and ${\mathrm{pr}}_1^\ast K\otimes^L(f\times
\mathrm{id}_{{\mathbb P}_{(\infty')}^1})^\ast \overline{\mathscr
L_\psi(\pi t')}$ vanish on the fibers over $\infty'$, their
vanishing cycles complex and nearby cycles complex coincide. So the
stalk at $(s,\infty')$ of the composite of the above canonical
isomorphisms can be identified with
$$R\Phi_{\eta_{\infty'}}\Big(p_1^\ast Rf_\ast K\otimes^L\overline{\mathscr
L_\psi(\pi t')}\Big)_{(s,\infty')}\cong
\left(R(f\times\mathrm{id}_{\overline\infty'})_\ast
R\Phi_{\eta_{\infty'}} \Big({\mathrm{pr}}_1^\ast K\otimes^L(f\times
\mathrm{id}_{{\mathbb P}_{(\infty')}^1})^\ast \overline{\mathscr
L_\psi(\pi t')}\Big)\right)_{(s,\infty')}.$$ Applying Lemma
\ref{techlemma} to the complex $Rf_\ast K$ on $S$, we get
$$R\Phi_{\eta_{\infty'}}\Big(p_1^\ast Rf_\ast K\otimes^L\overline{\mathscr
L_\psi(\pi t')}\Big)\cong R\Phi_{\eta_{\infty'}}\Big(p_1^\ast
j_!R\Phi(Rf_\ast K)\otimes^L\overline{\mathscr L_\psi(\pi
t')}\Big).$$ So we get a canonical isomorphism
\begin{equation}\label{iso(3)}
\begin{split}{}
&R\Phi_{\eta_{\infty'}}\Big(p_1^\ast j_! R\Phi \big(Rf_\ast K\big)
\otimes^L\overline{\mathscr L_\psi(\pi t')}\Big)_{(s,\infty')}\\
\cong\;& \left(R(f\times\mathrm{id}_{\overline\infty'})_\ast
R\Phi_{\eta_{\infty'}} \Big({\mathrm{pr}}_1^\ast K\otimes^L(f\times
\mathrm{id}_{{\mathbb P}_{(\infty')}^1})^\ast \overline{\mathscr
L_\psi(\pi t')}\Big)\right)_{(s,\infty')}.
\end{split}
\end{equation}
By the proper base change theorem and lemma \ref{prelemma}, we have
\begin{eqnarray*}
R\Phi \big(Rf_\ast K\big)&\cong& Rf_{s\ast}(R\Phi_{f}(K))\\
&\cong& \bigoplus_{x\in A} (R\Phi_{f}(K))_{x}.
\end{eqnarray*}
By Lemma \ref{keylemma} (i) that we have shown at the beginning of
this section, we have
\begin{eqnarray*}
&&\left(R(f\times\mathrm{id}_{\overline\infty'})_\ast
R\Phi_{\eta_{\infty'}} \Big({\mathrm{pr}}_1^\ast K\otimes^L(f\times
\mathrm{id}_{{\mathbb
P}_{(\infty')}^1})^\ast \overline{\mathscr L_\psi(\pi t')}\Big)\right)_{(s,\infty')}\\
&\cong& \bigoplus_{x\in A} R\Phi_{\eta_{\infty'}}
\Big({\mathrm{pr}}_1^\ast K\otimes^L(f\times \mathrm{id}_{{\mathbb
P}_{(\infty')}^1})^\ast \overline{\mathscr L_\psi(\pi
t')}\Big)_{(x,\infty')}.
\end{eqnarray*}
So we can write the isomorphism (\ref{iso(3)}) as
\begin{eqnarray*}
&&\bigoplus_{x\in A} R\Phi_{\eta_{\infty'}}\Big(p_1^\ast j_!
(R\Phi_{f}(K))_{x}
\otimes^L\overline{\mathscr L_\psi(\pi t')}\Big)_{(s,\infty')}\\
&\cong& \bigoplus_{x\in A}
R\Phi_{\eta_{\infty'}}\Big({\mathrm{pr}}_1^\ast K\otimes^L(f\times
\mathrm{id}_{{\mathbb P}_{(\infty')}^1})^\ast \overline{\mathscr
L_\psi(\pi t')}\Big)_{(x,\infty')}.
\end{eqnarray*} This last isomorphism coincides with the morphism
induced by the canonical morphisms (\ref{iso(2)})
$$R\Phi_{\eta_{\infty'}}\Big(p_1^\ast j_!
(R\Phi_{f}(K))_{x} \otimes^L\overline{\mathscr L_\psi(\pi
t')}\Big)_{(s,\infty')}\to
R\Phi_{\eta_{\infty'}}\Big({\mathrm{pr}}_1^\ast K\otimes^L(f\times
\mathrm{id}_{{\mathbb P}_{(\infty')}^1})^\ast \overline{\mathscr
L_\psi(\pi t')}\Big)_{(x,\infty')}.$$ It follows that these
canonical morphisms are isomorphisms.
\end{proof}

We are now ready to prove Lemma \ref{keylemma} (ii) and (iv).

\begin{proof}[Proof of Lemma \ref{keylemma} (ii) and (iv)]
The canonical morphism (\ref{iso(2)}) defined before Lemma
\ref{globallemma}
$$R\Phi_{\eta_{\infty'}}\Big(p_1^\ast
\big(j_!(R\Phi_g(K))_{y}\big)\otimes^L \overline{\mathscr L_\psi(\pi
t')}\Big)_{(s,\infty')} \to
R\Phi_{\eta_{\infty'}}\Big({\mathrm{pr}}_1^\ast K\otimes^L(g\times
\mathrm{id}_{{\mathbb P}_{(\infty')}^1})^\ast \overline{\mathscr
L_\psi(\pi t')}\Big)_{(y,\infty)}$$ commutes with the base change on
the trait $S$. By \cite[Th. finitude 3.7]{DSGA}, to prove this
canonical morphism is an isomorphism, we can reduce to the case
where $k$ is algebraically closed and $S=\mathrm{Spec}\,k[[\pi]]$.
By \cite[XVI 2.5]{DK}, there exist a smooth projective variety $Z$
over $k$, a morphism $h:Z\to \mathbb P_k^1$ smooth outside a finite
set of closed points in $Z$, and a closed point $z$ of $Z$ such that
$h(z)=0$ and that there exists an $S$-isomorphism between the
henselization of $Y$ at $y$, and the henselization of
$Z\times_{\mathbb P_k^1} S$ at $z$, where $S$ is regarded as a
${\mathbb P}_k^1$-scheme by identifying $S$ with the completion of
$\mathbb P_k^1$ at $0$. To prove the canonical morphism (3) is an
isomorphism is a problem local with respect to the etale topology.
So we may reduce to the case where $Y$ is an open subscheme of
$X=Z\times_{\mathbb P_k^1}S$, $y=z$, and $g:Y\to S$ is induced by
the projection $f:X=Z\times_{\mathbb P_k^1}S\to S$. The morphism
$f:X\to S$ is projective, and is smooth at points in $f^{-1}(s)-A$
for a finite set $A$ of closed points in $f^{-1}(s)$. Our assertion
then follows from Lemma \ref{globallemma}.
\end{proof}

\section{Questions and conjectures}

Deligne makes the following conjecture (\cite{DLetterFu}):

\begin{conjecture} Let
$\Lambda$ be a noetherian torsion ring with the property that
$m\Lambda=0$ for some integer $m$ relatively prime to $p$. let
$S_1$, $S_2$, $S$ be henselian traits over $\mathrm{Spec}\,k$ of
equal characteristic $p$, all with residue field $k$, let $s_1$,
$s_2$, $s$ be their closed points, and let $a:S_1\times_k S_2\to S$
be a $k$-morphism such that $a(s_1,s_2)=s$ and such that $a(\cdot,
s_2):S_1\to S$ and $a(s_1,\cdot):S_2\to S$ are isomorphisms. For
each $i$, let $f_i:X_i\to S_i$ be a morphism of finite type, let
$x_i$ be a $k$-rational point in the special fiber such that
$f_i|_{X_i-x_i}$ is smooth, and let
$K_i\in\mathrm{ob}\,D^b_{ctf}(X_i,\mathbb Z/\ell^m)$ such that
$\mathscr H^q(K_i)|_{X_i-x_i}$ are locally constant for all $q$. (Or
more optimistically, one can just assume $K_i$ is universally
locally acylic outside $x_i$ relative to $f_i$.) Let $\eta_i$ be the
generic point of $S_i$, and let $j_i:\eta_i\hookrightarrow S_i$ be
the canonical open immersion.

(i) Under the above conditions, $a\circ(f_1\times f_2)$ is locally
acyclic relative to $K_1\boxtimes^LK_2$ outside $(x_1,x_2)$, and
hence the vanishing cycles complex $R\Phi_{a\circ(f_1\times
f_2)}(K_1\boxtimes^L K_2)$ relative to the morphism
$a\circ(f_1\times f_2)$ is supported at $(x_1,x_2)$. (Confer
\cite[Th. finitude 2.12]{DSGA} for the definition of local
acyclicity.)

(ii) There is a canonical isomorphism
\begin{eqnarray*}
&&R\Phi_{a}\Big(j_{1!}\big((R\Phi_{f_1}(K_1))_{x_1}\big)\boxtimes^L
j_{2!}\big((R\Phi_{f_2}(K_2))_{x_2}\big)\Big)_{(s_1,s_2)} \\
&\cong& \Big(R\Phi_{a\circ(f_1\times f_2)}(K_1\boxtimes^L
K_2)\Big)_{(x_1,x_2)},
\end{eqnarray*} where $(R\Phi_{f_i}(K_i))_{x_i}$
($i=1,2$) are complexes of $\Lambda$-modules with
$\mathrm{Gal}(\bar\eta_i/\eta_i)$-action and are regarded as objects
in $D^b_{ctf}(\eta_i, \Lambda)$.
\end{conjecture}

Due to the use of the Deligne-Fourier transformation, the method
used in this paper is not applicable to solve the conjecture unless
$a$ is directly related to the addition of the algebraic group
$\mathbb A_k^1$. Moreover, our method works only in characteristic
$p$. Another problem that is not treated in the paper is how the
Thom-Sebastiani theorem behaves with respect to the variation
morphism. This problem over $\mathbb C$ is treated by Deligne
\cite{DLetter} by topological methods.

\end{document}